\documentclass[sn-mathphys,Numbered]{sn-jnl}
\usepackage[caption=false]{subfig}
\usepackage{graphicx}%
\usepackage{multirow}%
\usepackage{amsmath,amssymb,amsfonts}%
\usepackage{amsthm}%
\usepackage{mathrsfs}%
\usepackage[title]{appendix}%
\usepackage{xcolor}%
\usepackage{textcomp}%
\usepackage{manyfoot}%
\usepackage{booktabs}%
\usepackage{algorithm}%
\usepackage{algorithmicx}%
\usepackage{algpseudocode}%
\usepackage{listings}%

\newtheorem{theorem}{Theorem}
\newtheorem{remark}{Remark}%
\newtheorem{corollary}{Corollary}
\newtheorem{lemma}{Lemma}

\newtheorem{definition}{Definition}%
\newtheorem{assumption}{Assumption}

\usepackage{tikz}
\usetikzlibrary{decorations.pathreplacing}
\usepackage{circuitikz}
\usepackage{multirow}
\usepackage{color}
\definecolor{NewOrange}{RGB}{249,200,120}
\definecolor{NewBlue}{RGB}{140,185,237}
\definecolor{Newwhite}{RGB}{255,255,255}
\usepackage{subfig}
\usepackage{float}
\usepackage{multirow}
\usepackage{caption}

\usepackage[justification=centering]{caption}
\begin{document}
%
\title[Article Title]{Nash Equilibrium Seeking in Networked Games with Intermittent Communication}
%
%
\author{\fnm{Ying} \sur{Zhai}}\email{zhaiyinghit@163.com}

\author{\fnm{Rui} \sur{Yuan}}\email{yuanrui1580@gmail.com}

\author*{\fnm{Huan} \sur{Su}}\email{suhuan1981@163.com}

\affil{\orgdiv{Department of Mathematics}, \orgname{Harbin Institute of Technology}, \orgaddress{\street{Wenhua West Road}, \city{Weihai}, \postcode{264209}, \state{Shandong Province}, \country{China}}}

\abstract{This paper investigates the Nash equilibrium seeking problems for networked games with intermittent communication, where each player is capable of communicating with other players intermittently over a strongly connected and directed graph. Noticing that the players are not directly and continuously available for the actions of other players, this paper proposed an intermittent communication strategy. Compared with previous literature on intermittent communication, the players considered in this paper communicate with other players without quasi-periodic constraint. Instead, the players are supposed to estimate the actions of the other players with completely aperiodically intermittent communication. The distributions of communication time and silent time are characterized newly according to the concept of average communication ratio. And each player estimates other players' actions only during communication time. Finally, the validity of the theoretical results is demonstrated by two simulation examples.}

\keywords{Nash equilibrium seeking, Networked games, Aperiodically intermittent communication, Average communication ratio}
\maketitle
\section{Introduction}\label{sec1}

In the past few years, game theory has penetrated into various research areas, from robot networks, cognitive networks and smart grids, to economic markets (e.g., see \cite{ye2017distributed,frihauf2011nash,koshal2016distributed,ratliff2016characterization}). As one of the significant issues, Nash equilibrium seeking strategies for networked games have drawn some attention recently. With the incentive to solve Nash equilibrium seeking problems for networked games, many strategies have been exploited (e.g., see \cite{WOS,AII,Pu_2022,deng2021distributed,deng2022nash,li2014designing,ye2021distributed,stankovic2011distributed,ye2019rise} and the references therein).

It should be noted that games with limited information appear early and widely, but most of the existing works still concentrate on avoiding full communication among the players. In \cite{koshal2016distributed}, games with limited information flow were proposed. The authors in \cite{ye2017distributed} formulated distributed Nash equilibrium seeking as a leader-following consensus problem for noncooperative games. A hybrid gradient search algorithm was used for seeking Nash equilibrium with continuous-time communication in \cite{ye2020distributed}. The second-order integrator-type players were studied in \cite{ye2021distributed} for games with velocity estimation. An event-triggered law was utilized to design Nash equilibrium strategies for noncooperative games \cite{zhang2021distributed}. The adaptive approaches developed in \cite{ye2021adaptive} allows each player to adjust its own weight on its procurable consensus error dynamically. However, limited information includes not only local communication between different players, but also intermittent communication. In fact, the information transmission among the players may be intermittent. For example, players are not easy to get the full information from other players in the situation of constrained resource. Therefore, it is more realistic and reasonable to take the resource-constrained circumstances into consideration when designing a Nash equilibrium seeking strategy. To this end, the authors solved the noncooperative game based on a distributed discrete-time algorithm in \cite{zhang2021nash}, where each player estimates and exchanges the estimate information only when the event-triggered conditions are satisfied. In addition, extremum seeking methods were presented to seek the Nash equilibrium in \cite{krilavsevic2021learning,poveda2017framework,liu2011stochastic}. Although the methods stated above vary from a continuous-time algorithm \cite{krilavsevic2021learning}, to a discrete-time algorithm \cite{ye2020extremum}. A common characteristic of these algorithms is transmitting the information all the time.

Actually, in \cite{wang2021limited}, the authors have tried to determine the impacts of periodically intermittent communication on the limited-budget consensus problem. The authors in \cite{wang2018cooperative} showed that the consensus problem for a class of multi-agent systems with periodically intermittent communication could be cast into the stabilization problem of a set of simpler state-switching systems. In fact, communication period and communication time width have fixed constants in \cite{wang2021limited} and \cite{wang2018cooperative}, which is not flexible enough. Furthermore, the method in \cite{chen2022finite} was designed for an autonomous underwater vehicle system with aperiodically intermittent communication. The constrained consensus protocol for multi-agent systems with aperiodically intermittent communication was proposed in \cite{wen2014distributed}. Similarly, although intermittent communication in \cite{chen2022finite} and \cite{wen2014distributed} is aperiodic, it is supposed to satisfy the quasi-periodicity assumption, which is unreasonable. In a word, most of the existing literature about aperiodically intermittent communication require each communication time width should be larger than or equal to a common constant to ensure the convergence of studied systems, which can be summarized as follows.

\begin{itemize}
\item[1)] Based on the quasi-periodicity assumption, both communication period and communication time are limited by constants. (e.g., see \cite{chen2022finite,wen2014distributed,liu2015synchronization})

\item[2)] Each distribution of communication time is characterized by a common minimum communication ratio. (e.g., see \cite{liu2014synchronization,dong2022intermittent})
\end{itemize}

In this regard, the methods mentioned above are not completely aperiodic. Moreover, these methods are conservative due to uniformly distributed communication time on the one hand. On the other hand, they are difficult to implement in engineering systems where intermittent communication should be guaranteed. In fact, the communication between players should be available at any time and for any durations. In order to solve games with intermittent communication, in this paper, two traditional cases of intermittent communication are used to design Nash equilibrium seeking strategies. Furthermore, the average communication ratio is leveraged for characterizing the distributions of communication time and silent time. That is to say, the objective of this paper is to make the lower bound of some communication time be smaller than methods in \cite{chen2022finite,liu2015synchronization}, the upper bound of some communication period be larger than methods in \cite{wen2014distributed}, which is less conservative compared with 1). The proportion of silent time can be any value in $(0,1)$, which is more reasonable compared with 2).

The main idea of this paper is to investigate networked games with intermittent communication over a strongly connected and directed graph, as well as developing a novel intermittent communication strategy to seek Nash equilibrium. Compared with the existing works, the core contributions and novelties of this paper are summarized as follows.

\begin{itemize}
\item[1)] A theoretical framework for Nash equilibrium seeking strategy for networked games with intermittent communication is established. The games with limited information are quite comprehensive that they involve intermittent communication and avoid full communication among the players simultaneously.
\item[2)] This paper reveal that the convergence property of networked games can be achieved by proposed strategy. Moreover, the proposed strategy requires that the players estimate information with other players only during communication time. As such, the proposed strategy is more cost effective.
\item[3)] An aperiodically intermittent communication strategy is designed to realize the seeking of Nash equilibrium. It characterizes the distributions of communication time and silent time in a new form by the concept of average communication ratio, which relaxes the constraints on communication period and communication time.

\end{itemize}

The organization of this paper is as follows. Section \ref{sec2} offers preliminary knowledge and formulates the networked game. Section \ref{sec3} shows three Nash equilibrium seeking strategies for the networked game with intermittent communication. Section \ref{sec4} gives a simulation example to ensure the feasibility of the proposed strategy. Finally, Section \ref{sec5} summarizes the conclusion.

\emph{Notations:}
$\mathbb{R}$ and $\mathbb{R}^{N}$ are the set of real numbers and the $N$-dimensional Euclidean space, respectively. $\mathbf{0}_{N\times N}$ and $I_N$ represent $N\times N$ matrix with all elements equal to $0$ and $N\times N$ identity matrix, respectively. $\mathbf{0}_{N}$ and $\mathbf{1}_N$ represent $N$-dimensional column vector with all elements equal to $0$ and $1$, respectively. $\mathrm{diag}\{\cdot\}$ denotes a block-diagonal matrix. The symbol ``$\otimes$" is Kronecker product. The symbol ``$\ast$" indicates a symmetric structure in matrix expressions. Given a vector or matrix, $|\cdot|$ is the Euclidean norm of the vector, while $\Vert\cdot\Vert$ is the induced matrix norm. For a symmetric matrix $A$, $A > 0$ means that $A$ is a positive definite matrix. $\lambda_{\max}(A)$(respectively, $\lambda_{\min}(A)$) are the maximum (respectively, minimum) eigenvalue of $A$.

\section{Problem Fomulation}\label{sec2}
\subsection{Graph Theory}
Consider the networked game with $N$ players which are equipped with a communication graph $\mathcal{G}=(\mathcal{N},\mathcal{E})$, in which $\mathcal{N}=\{1, 2, \cdots, N\}$ is the vertex set of players, $\mathcal{E}\subset\mathcal{N}\times\mathcal{N}$ is the associated edge set. $\mathcal{A}=(a_{ij})_{N\times N}$ is the adjacency matrix where $a_{ij}>0$ if node $(j,i)\in\mathcal{E}$, together with $a_{ij}=0$ if node $(j,i)\notin\mathcal{E}$. Accordingly, the Laplacian matrix of $\mathcal{G}$ is $\mathcal{L}=\mathcal{D}-\mathcal{A}$ with $\mathcal{D}=\mathrm{diag}\{d_1^{in}, d_2^{in}, \cdots, d_N^{in}\}$, where $d_i^{in}=\sum_{j=1}^N a_{ij}$. Moreover, this paper considers that $\mathcal{G}$ is strongly connected in the sense that each node has at least one path to the other nodes. In addition, we consider that any self-loops are not allowed herein, i.e., $a_{ii}=0$, for $i\in\mathcal{N}$. With the strongly connected and directed graph $\mathcal{G}$, the following assumption and lemma can be used in the rest of this paper.

\begin{assumption}\label{ass1}
The directed graph $\mathcal{G}$ is strongly connected.
\end{assumption}

\begin{lemma}\label{lm1}(\cite{zhang2021distributed})
Under Assumption \ref{ass1}, there exist two positive definite matrices $P=\mathrm{diag}\{p_i\}$ and $Q$, such that for $i,j\in\mathcal{N}$
\begin{equation*}
    (\mathcal{L} \otimes I_N + \mathcal{B})^{\top}P+P(\mathcal{L} \otimes I_N + \mathcal{B})=Q,
\end{equation*}
where $\mathcal{B}=\mathrm{diag}\{a_{ij}\}$.
\end{lemma}

\subsection{Game Theory}
In the networked game, each player purposes to\\[0.2cm]
$~~~~~~~~~~~~~~~~~~~~~~~~~~\max_{x_i}~~f_i(\mathbf{x})$,\\[0.2cm]
where $x_i\in\mathbb{R}$ and $f_i(\mathbf{x})$ are the action and the payoff function of player $i$, respectively. This paper aims to design Nash equilibrium seeking strategy under three cases of intermittent communication, so that the players' actions can be convergent at the Nash equilibrium in the end. The definition is given as follows.

\begin{definition}\label{dn1}
Let $\mathbf{x}=(x_1,x_2,\cdots,x_N)^{\top}\in\mathbb{R}^{N}$ be the players' actions. $\mathbf{x}^*=(x_i^*,\mathbf{x}_{-i}^*)$ is called a Nash equilibrium of the game if
\begin{equation*}
    f_i(x_i,\mathbf{x}_{-i}^*) \leq f_i(x_i^*,\mathbf{x}_{-i}^*),
\end{equation*}
where for $i\in\mathcal{N}$, $\mathbf{x}_{-i}=(x_1,\cdots,x_{i-1},x_{i+1},\cdots,x_N)^{\top}$.
\end{definition}

Note that $f_i(x_i,\mathbf{x}_{-i})$ can be alternatively denoted as $ f_i(\mathbf{x})$ by rewriting $(x_i,\mathbf{x}_{-i})$ to $\mathbf{x}$ in this paper. Based on Definition \ref{dn1}, it is also revealed that no player can further increase its payoff by only changing its own action. The following assumptions and lemma will be utilized for the Nash equilibrium exploration.

\begin{assumption}\label{ass2}
Let $\mathcal{C}^2$ be the family of all continuous functions, which are twice-continuously differentiable in their own arguments. The payoff function of each player $f_i(\mathbf{x})\in\mathcal{C}^2$, for $i\in\mathcal{N}$.
\end{assumption}

\begin{assumption}\label{ass3}
For $i\in\mathcal{N}$, if $\frac{\partial f_i({\mathbf{x}})}{\partial x_i}$ satisfies that
\begin{equation*}
    \left|\frac{\partial f_i({\mathbf{x}})}{\partial x_i}-\frac{\partial f_i({\mathbf{z}})}{\partial z_i}\right|\leq\alpha|\mathbf{x}-\mathbf{z}|,
\end{equation*}
where $\mathbf{x}$, $\mathbf{z}\in\mathbb{R}^{N}$. $\frac{\partial f_i({\mathbf{x}})}{\partial x_i}$ is globally Lipschitz with constant $\alpha>0$.
\end{assumption}

\begin{assumption}\label{ass4}
For $i\in\mathcal{N}$, the payoff function of each player $f_i(\mathbf{x})$ is concave in $x_i$. Furthermore, for $\mathbf{x}$, $\mathbf{z}\in\mathbb{R}^{N}$,
\begin{equation*}
    (\mathbf{x}-\mathbf{z})^{\top}\left(\frac{\partial G}{\partial\mathbf{x}}(\mathbf{x})-\frac{\partial G}{\partial\mathbf{z}}(\mathbf{z})\right)\leq-\beta|\mathbf{x}-\mathbf{z}|^2,
\end{equation*}
where $\frac{\partial G}{\partial\mathbf{x}}(\mathbf{x})=(\frac{\partial f_1(\mathbf{x})}{\partial x_1},\frac{\partial f_2(\mathbf{x})}{\partial x_2},\cdots,\frac{\partial f_N(\mathbf{x})}{\partial x_N})^{\top}$, and constant $\beta>0$.
\end{assumption}

\begin{remark}\label{rk1}
In fact, Assumptions \ref{ass2}-\ref{ass4} ensure the characterizations on the existence and uniqueness of the Nash equilibrium, which are common in the existing works (e.g., see \cite{rosen1965existence,ye2016game,ye2021adaptive}). In addition, $\frac{\partial G}{\partial\mathbf{x}}(\mathbf{x}^*)=\mathbf{0}_{N}$ is called the stationary condition, if and only if $\mathbf{x}=\mathbf{x}^*$.
\end{remark}

For the convenience of the following description of lemma, let
\begin{equation*}
    \dot{x}_i=\bar{k}_i\frac{\partial f_i(\mathbf{x})}{\partial x_i},
\end{equation*}
which is called auxiliary function for $i\in\mathcal{N}$.
\begin{lemma}\label{lm2}(\cite{ye2017distributed})
Under Assumptions \ref{ass2}-\ref{ass4}, there exist some positive constants $\gamma_j$ $(j=1,2,3,4,5)$ such that
\begin{equation*}
\begin{split}
\gamma_1|\mathbf{e}|^2\leq\bar{V}(\mathbf{e})\leq&\gamma_2|\mathbf{e}|^2,\\
\left(\frac{\partial\bar{V}(\mathbf{e})}{\partial\mathbf{e}}\right)^{\top}\left(\mathrm{diag}\{\bar{k}_i\}\frac{\partial G}{\partial \mathbf{x}}(\mathbf{x})\right)\leq&-\gamma_3|\mathbf{e}|^2,\\
\left|\frac{\partial\bar{V}(\mathbf{e})}{\partial\mathbf{e}}\right|\leq&\gamma_4|\mathbf{e}|,
\end{split}
\end{equation*}
where $\bar{V}:\mathbb{D}\longmapsto\mathbb{R}$ is a function with $\mathbb{D}=\{\mathbf{x}\in\mathbb{R}^{N}|~|\mathbf{e}|\leq \gamma_5\}$, and $\mathbf{e}=\mathbf{x}-\mathbf{x}^*$.
\end{lemma}

\subsection*{C. Intermittent Communication}
In many practical applications, due to network bandwidth, equipment failures, technical constraints, communication costs and other reasons, the players is not communicating all the time, and there may be interrupted or intermittent communication transmission mode. For example, intermittent power supply, intermittent communication packet loss, Intermittent heating or cooling, and so on. Therefore, how to realize the seeking of Nash equilibrium in the networked game constrained by intermittent communication has become the focus of research.

This paper aims to design Nash equilibrium seeking strategies with intermittent communication, so that the players can estimate information with other players only during communication time. Actually, two kinds of intermittent communication laws have been proposed (e.g., see \cite{wang2018cooperative,wang2021limited,wen2014distributed,chen2022finite}). Inspired by this, the networked game with intermittent communication is proposed in this paper for the first time. Consider the auxiliary system, for $i,j\in\mathcal{N}$,
\begin{equation*}
\left\{
\begin{aligned}
\dot{y}_{ij}(t)=& -\left(\sum_{k=1}^{N}a_{ik}(y_{ij}-y_{kj})+a_{ij}(y_{ij}-x_j)\right),t\in[t_m,s_m),\\
\dot{y}_{ij}(t)=& ~0,t\in[s_m,t_{m+1}),
\end{aligned}
\right.
\end{equation*}
where $t_0=0$, $[t_m,s_m)$ is called the $m$-th communication time, $[s_m,t_{m+1})$ is called the $m$-th silent time, $m=\{0,1,2,\cdots\}$.
\begin{itemize}
  \item[$\bullet$] periodically intermittent communication (PIC)
\begin{enumerate}
\item If $s_m-t_m=\theta$, $t_{m+1}-t_m=T$, where $0<\theta<T<\infty$, then intermittent communication law becomes PIC as studied in \cite{wang2018cooperative}.
\end{enumerate}
 \item[$\bullet$] aperiodically intermittent communication (AIC)
\begin{enumerate}
\item (quasi-periodicity assumption, e.g., see \cite{chen2022finite,wen2014distributed,liu2015synchronization}) $\inf_m\{s_m-t_m\}=\theta$, $\sup_m\{t_{m+1}-t_m\}=T$, where $0<\theta<T<\infty$.
\item (minimum communication ratio assumption, e.g., see \cite{liu2014synchronization,dong2022intermittent}) $\limsup_{m\rightarrow\infty}\{\frac{t_{m+1}-s_m}{t_{m+1}-t_m}\}=\zeta\in(0,1)$.
\end{enumerate}
\end{itemize}

\begin{remark}\label{rk2}
From the quasi-periodicity assumption, it is not difficult to find that PIC is a special case of AIC. This also implies that the above assumptions are imposed to restrict the width and ratio of each communication time, see Fig. \ref{paic}. In \cite{zhai2021stabilization}, we have shown that $s_m-t_m$ can take any value in $(0,1)$ such that the system is globally asymptotically stable. However, it is challenging for networked games with intermittent communication to realize the seeking of Nash equilibrium.
\end{remark}

\begin{figure}[!htb]
\begin{center} 
	\subfloat[AIC considered in \cite{chen2022finite} and \cite{wen2014distributed}]{
	\begin{tikzpicture}[line width =0.7pt][->=0.01pt,>=stealth'=0.01pt,shorten >=0.01pt,auto,node distance=0.1cm,thick=0.1pt,main node/.style={circle=0.01pt,fill=blue,draw=0.1pt,			 font=0.1pt,minimum size=10pt}]
		\filldraw[fill=NewBlue] (0,0) rectangle (2,0.5) ;
		\draw (0,0)node[below] {$t_{m-1}$};
		\filldraw[fill=Newwhite] (2,0) rectangle (3,0.5) ;
		\draw (2,-0.06) node[below] {$s_{m-1}$};
		\filldraw[fill=NewBlue] (3,0) rectangle (4.5,0.5) ;
		\draw (3,0) node[below] {$t_{m}$};
		\filldraw[fill=Newwhite] (4.5,0) rectangle (5.8,0.5) ;
		\draw (4.5,-0.06)node[below] {$s_{m}$};
		\filldraw[fill=NewBlue] (5.8,0) rectangle (7,0.5) ;
		\draw (5.8,0)node[below] {$t_{m+1}$};
		\filldraw[fill=Newwhite] (7,0) rectangle (8,0.5) ;
		\draw (7,-0.06)node[below] {$s_{m+1}$};
		\draw (8,0)node[below] {$t_{m+2}$};
		\filldraw[fill=NewBlue] (0,1.05) rectangle (1,1.55) ;
		\filldraw[fill=Newwhite] (5,1.05) rectangle (6,1.55) ;
		\node[] at (2.7,1.3) {communication time};
		\node[] at (7,1.3) {silent time};
	\end{tikzpicture}
}

   \subfloat[AIC presented in this paper]{
    \begin{tikzpicture}[line width =0.7pt][->=0.01pt,>=stealth'=0.01pt,shorten >=0.01pt,auto,node distance=0.1cm,thick=0.1pt,main node/.style={circle=0.01pt,fill=blue,draw=0.1pt,			 font=0.1pt,minimum size=10pt}]
    	\filldraw[fill=NewBlue] (0,0) rectangle (2.7,0.5) ;
    	\draw (0,0)node[below] {$t_{m-1}$};
    	\filldraw[fill=Newwhite] (2.7,0) rectangle (3.6,0.5) ;
    	\draw (2.7,-0.06) node[below] {$s_{m-1}$};
    	\filldraw[fill=NewBlue] (3.6,0) rectangle (3.9,0.5) ;
    	\draw (3.5,0) node[below] {$t_{m}$};
    	\filldraw[fill=Newwhite] (3.9,0) rectangle (5.2,0.5) ;
    	\draw (4,-0.06)node[below] {$s_{m}$};
    	\filldraw[fill=NewBlue] (5.2,0) rectangle (6.8,0.5) ;
    	\draw (5.2,0)node[below] {$t_{m+1}$};
    	\filldraw[fill=Newwhite] (6.8,0) rectangle (8,0.5) ;
    	\draw (6.8,-0.06)node[below] {$s_{m+1}$};
    	\draw (8,0)node[below] {$t_{m+2}$};    		
    \end{tikzpicture}
}
\caption{Sketch map of intermittent communication.}\label{paic}
\end{center} 
\end{figure}
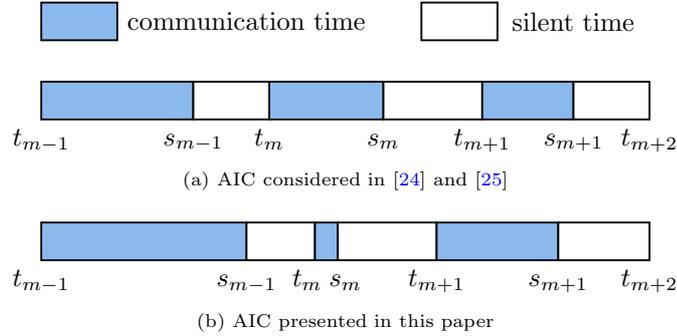

\begin{remark}\label{rk3}
To reduce the communication cost, the Nash equilibrium seeking strategy was presented in \cite{zhang2021distributed} by event-triggered communication, which only requires the communication signal to be updated when the triggering conditions are satisfied. In short, let $\theta\rightarrow T$ in PIC, and $T$ is the minimum triggering interval. That is to say, the communication time among players exists all the time. Different from \cite{zhang2021distributed,zhang2021nash}, the intermittent communication strategy allows the communication to disappear during some non-zero intervals, i.e., silent time.
\end{remark}

Motivated by the intention of relaxing the restrictions on AIC, this paper aims to devise a novel intermittent communication strategy from the average sense to achieve seeking of the Nash equilibrium. Subsequently, the following definition are introduced to illustrate average communication ratio (ACR).

\begin{definition}\label{dn2}
For AIC, suppose that there exists a constant $\vartheta\in(0,1)$ such that
\begin{equation*}
\begin{split}
    &\mathcal{M}(t,s)\geq\vartheta(t-s),\\
    &\mathcal{M}^{\mathrm{c}}(t,s)=t-s-\mathcal{M}(t,s)\geq(1-\vartheta)(t-s),
\end{split}
\end{equation*}
where $0\leq s<t<\infty$, $\mathcal{M}(t,s)$ stands for the sum of total communication time width on $[s,t)$, $\mathcal{M}^{\mathrm{c}}(t,s)$ stands for the sum of total silent time width on $[s,t)$.
\end{definition}

In this regard, we say that AIC considered in this paper admits an ACR $\vartheta\in(0,1)$. Additionally, in order to avoid no communication input in $[s,t)$, the elastic coefficient can be introduced in Definition \ref{dn2}, which can be seen in \cite{wang2021stabilization} and \cite{wang2021stability}.

\begin{remark}\label{rk4}
On the basis of Definition \ref{dn2}, it is revealed that the lower bound of some communication time width can be very small, instead of a common constant, see Fig. \ref{paic}. Hence, ACR adapts to a wider range of networked games with intermittent communication. Then it is obvious that the AIC based on Definition \ref{dn2} admits the complete aperiodicity, which is more reasonable.
\end{remark}

\section{Main Results} \label{sec3}
In this section, three Nash equilibrium seeking strategies will be proposed with intermittent communication. More specifically, the case of PIC, AIC with quasi-periodicity assumption, AIC with ACR will be presented, successively. And the relevant convergence analysis will be supplied based on Lyapunov stability analysis.

Motivated by \cite{ye2017distributed} and \cite{zhang2021distributed}, the action of each player $i$ is steered by
\begin{equation}\label{s1}
    \dot{x}_i(t)=k_i\frac{\partial f_i}{\partial x_i}(\mathbf{y}_i(t)),~~~~i\in\mathcal{N},
\end{equation}
where $k_i=\epsilon \bar{k}_i$, $\epsilon>0$ is a small scalar and $\bar{k}_i>0$ are some constants for $i\in\mathcal{N}$. Moreover, $\mathbf{y}_i\in\mathbb{R}^N$ stands for the estimates on all players¡¯ actions by player $i$ and is defined as $\mathbf{y}_i=(y_{i1},y_{12},\cdots,y_{iN})^{\top}$. Based on intermittent communication, $y_{ij}$ denotes the estimate on player $j$'s action by player $i$, which is produced by
\begin{equation}\label{icl}
\left\{
\begin{aligned}
\dot{y}_{ij}(t)=& -\left(\sum_{k=1}^{N}a_{ik}(y_{ij}-y_{kj})+a_{ij}(y_{ij}-x_j)\right),t\in[t_m,s_m),\\
\dot{y}_{ij}(t)=& ~0,t\in[s_m,t_{m+1}),
\end{aligned}
\right.
\end{equation}
where for $i,j\in\mathcal{N}$, $a_{ij}$ is the element the adjacency matrix of $\mathcal{G}$.

\begin{remark}\label{rk5}
During the communication time $[t_m,s_m)$, $y_{ij}$ can update in real time based on the estimates on the players' actions. During the silent time $[s_m,t_{m+1})$, $y_{ij}$ remain unchanged. Unlike many
existing studies on games where the Nash equilibrium seeking strategy was designed to be continuous communication (e.g., see \cite{deng2021distributed,deng2022nash,ye2019rise}), the intermittent estimates of the players are applied. Furthermore, during the silent time, the derivative of $y_{ij}$ is zero, instead of $y_{ij}=0$, which can make the strategy converge faster compared with previous works. Additionally, we introduce the concept of ACR, which can relax the constraints on conventional AIC and thus is less conservative.
\end{remark}

Under intermittent communication law \eqref{icl}, the Nash equilibrium seeking strategy can be designed as follows:
\begin{equation}\label{s2}
\begin{split}
\dot{\mathbf{x}}=&~\epsilon\mathrm{diag}\{\bar{k}_i\}\frac{\partial f}{\partial \mathbf{x}}(\mathbf{y}),\\
\dot{\mathbf{y}}(t)=&\left\{\begin{split}
&-(\mathcal{L}\otimes I_N+\mathcal{B})\mathbf{y}+\mathcal{B}(\mathbf{1}_N\otimes \mathbf{x}),t\in[t_m,s_m),\\
&~\mathbf{0}_{N\times N},t\in[s_m,t_{m+1}),
\end{split}
\right.
\end{split}
\end{equation}
where $\mathcal{B}$ is defined in Lemma \ref{lm1} and $\mathbf{y}=(\mathbf{y}_1^{\top},\mathbf{y}_2^{\top},\cdots,\mathbf{y}_N^{\top})^{\top}$. In the following, unless otherwise specified, $\mathbf{e}=\mathbf{x}-\mathbf{x}^*$ is defined as the consensus error, $\mathbf{e}_{\mathbf{x}}=\mathbf{y}-\mathbf{1}_N\otimes \mathbf{x}$ is defined as the estimate error, and $\frac{\partial f}{\partial \mathbf{x}}(\mathbf{1}_N\otimes \mathbf{x})$ is abbreviated as $\frac{\partial f}{\partial \mathbf{x}}(\mathbf{x})$. In fact, $\frac{\partial f}{\partial \mathbf{x}}(\mathbf{x})$ is the same as $\frac{\partial G}{\partial\mathbf{x}}(\mathbf{x})$ in Assumption \ref{ass4}. Then \eqref{s2} can be rewritten as
\begin{equation}\label{s3}
\dot{\mathbf{e}}=\dot{\mathbf{x}}=\epsilon\mathrm{diag}\{\bar{k}_i\}\frac{\partial f}{\partial \mathbf{x}}(\mathbf{e}_{\mathbf{x}}+\mathbf{1}_N\otimes \mathbf{x}),
\end{equation}
and for $t\in[t_m,s_m)$,
\begin{equation}\label{s4}
\begin{split}
\dot{\mathbf{e}}_{\mathbf{x}}(t)=&~\dot{\mathbf{y}}-\mathbf{1}_N\otimes \dot{\mathbf{x}}\\
=&-(\mathcal{L}\otimes I_N+\mathcal{B})(\mathbf{e}_{\mathbf{x}}+\mathbf{1}_N\otimes \mathbf{x})+\mathcal{B}(\mathbf{1}_N\otimes \mathbf{x})\\
&-\mathbf{1}_N\otimes\left(\epsilon\mathrm{diag}\{\bar{k}_i\}\frac{\partial f}{\partial \mathbf{x}}(\mathbf{e}_{\mathbf{x}}+\mathbf{1}_N\otimes \mathbf{x})\right),
\end{split}
\end{equation}
for $t\in[s_m,t_{m+1})$,
\begin{equation}\label{s5}
\begin{split}
\dot{\mathbf{e}}_{\mathbf{x}}(t)=&~\dot{\mathbf{y}}-\mathbf{1}_N\otimes \dot{\mathbf{x}}\\
=&-\mathbf{1}_N\otimes\left(\epsilon\mathrm{diag}\{\bar{k}_i\}\frac{\partial f}{\partial \mathbf{x}}(\mathbf{e}_{\mathbf{x}}+\mathbf{1}_N\otimes \mathbf{x})\right).~~~~
\end{split}
\end{equation}

Note that $\mathcal{G}$ is directed and strongly connected, so $\mathcal{L}\mathbf{1}_N=\mathbf{0}_N$ (\cite{olfati2006flocking}). Hence, we focus on the following strategy, derived by \eqref{s3}-\eqref{s5},
\begin{equation}\label{s6}
\begin{split}
\dot{\mathbf{e}}&=\epsilon\mathrm{diag}\{\bar{k}_i\}\frac{\partial f}{\partial \mathbf{x}}(\mathbf{e}_{\mathbf{x}}+\mathbf{1}_N\otimes \mathbf{x}),\\
\dot{\mathbf{e}}_{\mathbf{x}}(t)
&=\left\{\begin{split}
&-\mathcal{H}\mathbf{e}_{\mathbf{x}}-\mathbf{1}_N\otimes\left(\epsilon\mathrm{diag}\{\bar{k}_i\}\frac{\partial f}{\partial \mathbf{x}}(\mathbf{e}_{\mathbf{x}}+\mathbf{1}_N\otimes \mathbf{x})\right),t\in[t_m,s_m),\\
&-\mathbf{1}_N\otimes\left(\epsilon\mathrm{diag}\{\bar{k}_i\}\frac{\partial f}{\partial \mathbf{x}}(\mathbf{e}_{\mathbf{x}}+\mathbf{1}_N\otimes \mathbf{x})\right),t\in[s_m,t_{m+1}),
\end{split}
\right.
\end{split}
\end{equation}
where $\mathcal{H}=\mathcal{L}\otimes I_N+\mathcal{B}$. For notational convenience, let
\begin{equation*}
\begin{split}
\pi_1=&~\alpha\max\{\bar{k}_i\}(\frac{\gamma_4}{2}+N\Vert P\Vert),\\
\pi_2=&~2\alpha\sqrt{N}\Vert P\Vert\max\{\bar{k}_i\},\\
\epsilon_1^*=&~\frac{\gamma_3\lambda_{\min}(Q)}{\gamma_3\pi_2+\pi_1^2},\\
\epsilon_2^*=&~\frac{\gamma_3^2\pi_2^2}{\pi_1^4},\\
\epsilon^*=&~\min\{\epsilon_1^*,\epsilon_2^*\},\\
\Xi=&~(|\mathbf{e}|,\Vert\mathbf{e}_{\mathbf{x}}\Vert)^{\top},\\
\Gamma_1=&
  \left(
    \begin{array}{cc}
      \epsilon\gamma_3 & -\epsilon\pi_1 \\
      * &  \lambda_{\min}(Q)-\epsilon\pi_2\\
    \end{array}
  \right),\\
\Gamma_2=&
  \left(
    \begin{array}{cc}
      \sqrt{\epsilon}\gamma_3 & \epsilon\pi_1 \\
      * &  \epsilon\pi_2\\
    \end{array}
  \right),
\end{split}
\end{equation*}
where $\alpha$, $\gamma_3$, $\gamma_4$ are constants defined in Section \ref{sec2}, and $P$, $Q$ are positive definite matrices defined in Lemma \ref{lm1}. Then the following results can be obtained.

\subsection{Periodically Intermittent Communication}
In this section, we propose a periodic condition to limit the communication time width of \eqref{icl}. Set $t_m=m\tilde{T}$, $s_m=(m+\tilde{\theta})\tilde{T}$, and $t_{m+1}=(m+1)\tilde{T}$ in \eqref{icl}-\eqref{s6}, where $\tilde{T}$ is the communication period and $\tilde{\theta}\in(0,1)$ is the communication ratio. Then, \eqref{icl} operates periodically, and the following theorem is presented.

\begin{theorem}\label{tm1}
Assume that Assumptions \ref{ass1}-\ref{ass4} hold and the players' actions are governed by \eqref{s1} and \eqref{icl}. There exists a scalar $\epsilon\in(0,\epsilon^*)$, such that \eqref{s2} has the exponentially stable equilibrium point $(\mathbf{x}^*,\mathbf{1}_N\otimes\mathbf{x}^*)$ under PIC with communication ratio $\tilde{\theta}$ satisfying $\mu_1\tilde{\theta}-\mu_2(1-\tilde{\theta})>0$.
\end{theorem}

\begin{proof}
Take a Lyapunov function candidate as
\begin{equation}\label{t18}
V(\mathbf{e},\mathbf{e}_{\mathbf{x}})=\bar{V}(\mathbf{e})+\mathbf{e}_{\mathbf{x}}^{\top}P\mathbf{e}_{\mathbf{x}},
\end{equation}
where $\bar{V}(\mathbf{e})$ satisfies the obtained properties of Lemma \ref{lm2}, and $P$ is given in Lemma \ref{lm1}. For convenience, let $V=V(\mathbf{e},\mathbf{e}_{\mathbf{x}})$.
Obviously, there exist positive constants $\eta_1$ and $\eta_2$ such that
\begin{equation*}
\eta_1|\Xi|^2\leq V\leq\eta_2|\Xi|^2.
\end{equation*}

For $t\in[m\tilde{T},(m+\tilde{\theta})\tilde{T})$, taking the time derivative of $V$ along the trajectory of \eqref{s6} gives
\begin{equation}\label{t11}
\begin{split}
\dot{V}=&\left(\frac{\partial\bar{V}(\mathbf{e})}{\partial\mathbf{e}}\right)^{\top}\dot{\mathbf{e}}+2\mathbf{e}_{\mathbf{x}}^{\top}P\dot{\mathbf{e}}_{\mathbf{x}}\\
=&~\epsilon\left(\frac{\partial\bar{V}(\mathbf{e})}{\partial\mathbf{e}}\right)^{\top}\mathrm{diag}\{\bar{k}_i\}\frac{\partial f}{\partial \mathbf{x}}(\mathbf{e}_{\mathbf{x}}+\mathbf{1}_N\otimes \mathbf{x})\\
&-2\mathbf{e}_{\mathbf{x}}^{\top}P\left(\mathbf{1}_N\otimes\left(\epsilon\mathrm{diag}\{\bar{k}_i\}\frac{\partial f}{\partial \mathbf{x}}(\mathbf{e}_{\mathbf{x}}+\mathbf{1}_N\otimes \mathbf{x})\right)\right)
-2\mathbf{e}_{\mathbf{x}}^{\top}P\mathcal{H}\mathbf{e}_{\mathbf{x}}.
\end{split}
\end{equation}
According to Lemma \ref{lm2} and assumption \ref{ass3}, it follows that
\begin{equation}\label{t12}
\begin{split}
&~\epsilon\left(\frac{\partial\bar{V}(\mathbf{e})}{\partial\mathbf{e}}\right)^{\top}\mathrm{diag}\{\bar{k}_i\}\left(\frac{\partial f}{\partial \mathbf{x}}(\mathbf{e}_{\mathbf{x}}+\mathbf{1}_N\otimes \mathbf{x})-\frac{\partial f}{\partial \mathbf{x}}(\mathbf{x})\right)\\
\leq&~\epsilon\alpha\gamma_4\max\{\bar{k}_i\}|\mathbf{e}|\Vert\mathbf{e}_{\mathbf{x}}\Vert,
\end{split}
\end{equation}
and
\begin{equation}\label{t13}
\epsilon\left(\frac{\partial\bar{V}(\mathbf{e})}{\partial\mathbf{e}}\right)^{\top}\mathrm{diag}\{\bar{k}_i\}\frac{\partial f}{\partial \mathbf{x}}(\mathbf{x})\leq-\epsilon\gamma_3|\mathbf{e}|^2.
\end{equation}
Recalling $\frac{\partial f}{\partial\mathbf{x}}(\mathbf{x}^*)=\mathbf{0}_N$ and $\mathbf{e}=\mathbf{x}-\mathbf{x}^*$, we observe that
\begin{equation}\label{t14}
\begin{split}
&-2\mathbf{e}_{\mathbf{x}}^{\top}P\left(\mathbf{1}_N\otimes\left(\epsilon\mathrm{diag}\{\bar{k}_i\}\frac{\partial f}{\partial \mathbf{x}}(\mathbf{e}_{\mathbf{x}}+\mathbf{1}_N\otimes \mathbf{x})\right)\right)\\
\leq&~2\epsilon\alpha\sqrt{N}\Vert P\Vert\max\{\bar{k}_i\}\Vert\mathbf{e}_{\mathbf{x}}\Vert\Vert\mathbf{e}_{\mathbf{x}}+\mathbf{1}_N\otimes \mathbf{e}\Vert\\
\leq&~2\epsilon\alpha\sqrt{N}\Vert P\Vert\max\{\bar{k}_i\}\Vert\mathbf{e}_{\mathbf{x}}\Vert^2
+2\epsilon\alpha N\Vert P\Vert\max\{\bar{k}_i\}|\mathbf{e}|\Vert\mathbf{e}_{\mathbf{x}}\Vert.
\end{split}
\end{equation}
From Lemma \ref{lm1}, one has
\begin{equation}\label{t15}
-2\mathbf{e}_{\mathbf{x}}^{\top}P\mathcal{H}\mathbf{e}_{\mathbf{x}}=-\mathbf{e}_{\mathbf{x}}^{\top}Q\mathbf{e}_{\mathbf{x}}\leq-\lambda_{\min}(Q)\Vert\mathbf{e}_{\mathbf{x}}\Vert^2.
\end{equation}

By substituting \eqref{t12}-\eqref{t15} into \eqref{t11}, it yields
\begin{equation*}
\begin{split}
\dot{V}\leq&-\epsilon\gamma_3|\mathbf{e}|^2+2\epsilon\pi_1|\mathbf{e}|\Vert\mathbf{e}_{\mathbf{x}}\Vert+(\epsilon\pi_2-\lambda_{\min}(Q))\Vert\mathbf{e}_{\mathbf{x}}\Vert^2\\
=&-\Xi^{\top}\Gamma_1\Xi,
\end{split}
\end{equation*}
where $\Gamma_1$ is a positive definite matrix for $\epsilon\in(0,\epsilon_1^*)$. Consequently, we conclude that
\begin{equation}\label{t19}
\begin{split}
\dot{V}\leq&-\lambda_{\min}(\Gamma_1)|\Xi|^2\\
\leq&-\frac{\lambda_{\min}(\Gamma_1)}{\eta_2}V\\
\triangleq&-\mu_1V,
\end{split}
\end{equation}
which further means that for $t\in[m\tilde{T},(m+\tilde{\theta})\tilde{T})$,
\begin{equation}\label{t16}
V(t)\leq V(m\tilde{T})\exp\{-\mu_1(t-m\tilde{T})\}.
\end{equation}
Similarly, for $t\in[(m+\tilde{\theta})\tilde{T},(m+1)\tilde{T})$, we can obtain
\begin{equation*}
\begin{split}
\dot{V}
=&~\epsilon\left(\frac{\partial\bar{V}(\mathbf{e})}{\partial\mathbf{e}}\right)^{\top}\mathrm{diag}\{\bar{k}_i\}\frac{\partial f}{\partial \mathbf{x}}(\mathbf{e}_{\mathbf{x}}+\mathbf{1}_N\otimes \mathbf{x})\\
&-2\mathbf{e}_{\mathbf{x}}^{\top}P\left(\mathbf{1}_N\otimes\left(\epsilon\mathrm{diag}\{\bar{k}_i\}\frac{\partial f}{\partial \mathbf{x}}(\mathbf{e}_{\mathbf{x}}+\mathbf{1}_N\otimes \mathbf{x})\right)\right)\\
\leq&-\epsilon\gamma_3|\mathbf{e}|^2+2\epsilon\pi_1|\mathbf{e}|\Vert\mathbf{e}_{\mathbf{x}}\Vert+\epsilon\pi_2\Vert\mathbf{e}_{\mathbf{x}}\Vert^2\\
\leq&\sqrt{\epsilon}\gamma_3|\mathbf{e}|^2+2\epsilon\pi_1|\mathbf{e}|\Vert\mathbf{e}_{\mathbf{x}}\Vert+\epsilon\pi_2\Vert\mathbf{e}_{\mathbf{x}}\Vert^2\\
=&~\Xi^{\top}\Gamma_2\Xi,
\end{split}
\end{equation*}
where $\Gamma_2$ is a positive definite matrix for $\epsilon\in(0,\epsilon_2^*)$. We also conclude that
\begin{equation}\label{t110}
\begin{split}
\dot{V}\leq&~\lambda_{\max}(\Gamma_2)|\Xi|^2\\
\leq&~\frac{\lambda_{\min}(\Gamma_1)}{\eta_1}V\\
\triangleq&~\mu_2V,
\end{split}
\end{equation}
which further means that for $t\in[(m+\tilde{\theta})\tilde{T},(m+1)\tilde{T})$,
\begin{equation}\label{t17}
V(t)\leq V((m+\tilde{\theta})\tilde{T})\exp\{\mu_2(t-(m+\tilde{\theta})\tilde{T})\}.
\end{equation}

By the combination of \eqref{t16} and \eqref{t17}, using recursion yields, for $t\in[m\tilde{T},(m+\tilde{\theta})\tilde{T})$,
\begin{equation*}
\begin{split}
V(t)\leq&~V(m\tilde{T})\exp\{-\mu_1(t-m\tilde{T})\}\\
\leq&~V(0)\exp\{-m\mu_1\tilde{\theta} \tilde{T}+m\mu_2(1-\tilde{\theta})\tilde{T}\}\\
\leq&~V(0)\exp\{-(\mu_1\tilde{\theta}-\mu_2(1-\tilde{\theta}))(t-\tilde{\theta} \tilde{T})\}.
\end{split}
\end{equation*}
For $t\in[(m+\tilde{\theta})\tilde{T},(m+1)\tilde{T})$,
\begin{equation*}
\begin{split}
V(t)\leq&~V((m+\tilde{\theta})\tilde{T})\exp\{\mu_2(t-(m+\tilde{\theta})\tilde{T})\}\\
\leq&~V(0)\exp\{-(m+1)\mu_1\tilde{\theta} \tilde{T}+(m+1)\mu_2(1-\tilde{\theta})\tilde{T}\}\\
\leq&~V(0)\exp\{-(\mu_1\tilde{\theta}-\mu_2(1-\tilde{\theta}))(t-\tilde{\theta} \tilde{T})\}.
\end{split}
\end{equation*}
Hence, it holds that for any $t>0$, $\epsilon\in(0,\epsilon^*)$,
\begin{equation*}
\begin{split}
V(t)\leq V(0)\exp\{-(\mu_1\tilde{\theta}-\mu_2(1-\tilde{\theta}))(t-\tilde{\theta} \tilde{T})\},
\end{split}
\end{equation*}
which together with $\mu_1\tilde{\theta}-\mu_2(1-\tilde{\theta})>0$ implies that $(\mathbf{x}^*,\mathbf{1}_N\otimes\mathbf{x}^*)$ is exponentially stable.
\end{proof}

\subsection{Aperiodically Intermittent Communication with Quasi-periodicity Assumption}
In this section, we propose the quasi-periodicity assumption to limit the aperiodicity of the communication time width of \eqref{icl}. The quasi-periodicity assumption implies that each communication time width must be greater than or equal to $\bar{\theta}$, meanwhile, the sum of communication time width and silent time width must be less than or equal to $\bar{T}$. Thus, the communication time width must be no larger than $\bar{T}-\bar{\theta}$. This assumption has been widely applied in previous literature (e.g., see \cite{chen2022finite,wen2014distributed,liu2015synchronization,liu2014synchronization,dong2022intermittent}). In what follows, we propose an AIC strategy to realize the Nash equilibrium seeking.

\begin{theorem}\label{tm2}
Assume that Assumptions \ref{ass1}-\ref{ass4} hold and the players' actions are governed by \eqref{s1} and \eqref{icl}. There exists a scalar $\epsilon\in(0,\epsilon^*)$, such that \eqref{s2} has the exponentially stable equilibrium point $(\mathbf{x}^*,\mathbf{1}_N\otimes\mathbf{x}^*)$ under AIC with quasi-periodicity assumption
\begin{equation*}
\inf_m\{s_m-t_m\}=\bar{\theta},~~\sup_m\{t_{m+1}-t_m\}=\bar{T},
\end{equation*}
and $\rho=\mu_1\bar{\theta}-\mu_2(\bar{T}-\bar{\theta})>0$, where $0<\bar{\theta}<\bar{T}<\infty$.
\end{theorem}
\begin{proof}
Use the same Lyapunov function \eqref{t18} in Theorem \ref{tm1}. Similarly, for any $t>0$, $\epsilon\in(0,\epsilon^*)$, it follows from \eqref{t19} and \eqref{t110} that
\begin{equation}\label{t21}
\dot{V}(t)\leq\left\{
\begin{aligned}
&V(t_m)\exp\{-\mu_1(t-t_m)\},~~~t\in[t_m,s_m),\\
&V(s_m)\exp\{\mu_2(t-s_m)\},~~t\in[s_m,t_{m+1}),
\end{aligned}
\right.
\end{equation}
which further means that
\begin{equation}\label{t22}
\begin{split}
V(s_m)\leq&~V(t_m)\exp\{-\mu_1(s_m-t_m)\}\\
\leq&~V(t_m)\exp\{-\mu_1\bar{\theta}\},
\end{split}
\end{equation}
and
\begin{equation}\label{t23}
\begin{split}
V(t_{m+1})\leq&~V(t_m)\exp\{-\mu_1(s_m-t_m)+\mu_2(t_{m+1}-s_m)\}\\
\leq&~V(t_m)\exp\{-\mu_1\bar{\theta}+\mu_2(\bar{T}-\bar{\theta})\}.
\end{split}
\end{equation}
By recursion, one obtains
\begin{equation*}
V(t_{m+1})\leq V(0)\exp\{-m\rho\}\leq V(0)\exp\{\rho\}\exp\{-\frac{\rho}{\bar{T}}t\}.
\end{equation*}
It can be verified that when $\rho>0$, $\lim_{m\rightarrow\infty}V(t_{m+1})=0$. Then, associating with \eqref{t22} and \eqref{t23}, it is easy to obtain that $V(t)\leq V(0)\exp\{\rho\}\exp\{-\frac{\rho}{\bar{T}}t\}$ and $\lim_{t\rightarrow\infty}V(t)=0$. Accordingly, $(\mathbf{x}^*,\mathbf{1}_N\otimes\mathbf{x}^*)$ is exponentially stable.
\end{proof}

\begin{corollary}\label{cy1}
Assume that Assumptions \ref{ass1}-\ref{ass4} hold and the players' actions are governed by \eqref{s1} and \eqref{icl}. There exists a scalar $\epsilon\in(0,\epsilon^*)$, such that \eqref{s2} has the exponentially stable equilibrium point $(\mathbf{x}^*,\mathbf{1}_N\otimes\mathbf{x}^*)$ under AIC with minimum communication ratio assumption
\begin{equation*}
\limsup_{m\rightarrow\infty}\{\frac{t_{m+1}-s_m}{t_{m+1}-t_m}\}=\bar{\zeta}\in(0,1).
\end{equation*}
\end{corollary}
\begin{proof}
In this case, the conclusion of Theorem \ref{tm2} is still valid, where $\rho>0$ is replaced by
\begin{equation*}
\begin{split}
\rho_m=&~\mu_1(s_m-t_m)-\mu_2(t_{m+1}-s_m)\\
\geq&~(t_{m+1}-t_m)\left(\mu_1(1-\bar{\zeta})-\mu_2\bar{\zeta}\right)\\
>&~0.
\end{split}
\end{equation*}
Then it can obviously be acquired by the similar proof process of Theorem \ref{tm2} with $\hat{\rho}=\inf_m\{\rho_m\}>0$, i.e.,
\begin{equation*}
V(t)\leq V(0)\exp\{\hat{\rho}\}\exp\{-\frac{\hat{\rho}}{\bar{T}}t\}.
\end{equation*}
Then $\lim_{t\rightarrow\infty}V(t)=0$. Accordingly, $(\mathbf{x}^*,\mathbf{1}_N\otimes\mathbf{x}^*)$ is exponentially stable.
\end{proof}

\begin{remark}\label{rk6}
It should be pointed out that both continuous and intermittent Nash equilibrium seeking strategy can realize Nash equilibrium seeking. In particular, the continuous case can be regarded as a special case. That is to say, the proposed PIC and conventional AIC Nash equilibrium seeking strategies remove the requirement on continuous-time information exchanging among neighboring players. Compared with the continuous methods in \cite{ye2020distributed} and \cite{ma2014distributed}, it is clear that intermittent communication is more applicable in some practical problems.
\end{remark}

\subsection{ Aperiodically Intermittent Communication with Average Communication Ratio}
However, in some practical problems, the assumptions of quasi-periodicity and minimum communication ratio for AIC are unreasonable. Therefore, we should estimate approximately the average width of communication time in the whole time, which implies that considering the AIC from the average sense. Different from conventional AIC, in this section, we will consider the AIC with ACR. In this regard, the aperiodicity restriction in Theorem \ref{tm2} and Corollary \ref{cy1} vanishes, which illustrates that AIC is completely aperiodic. In the following, we will give the theorem and theoretical analysis.

\begin{theorem}\label{tm3}
Assume that Assumptions \ref{ass1}-\ref{ass4} hold and the players' actions are governed by \eqref{s1} and \eqref{icl}. There exists a scalar $\epsilon\in(0,\epsilon^*)$, such that \eqref{s2} has the exponentially stable equilibrium point $(\mathbf{x}^*,\mathbf{1}_N\otimes\mathbf{x}^*)$ under AIC with ACR $\vartheta\in(0,1)$ satisfying $\vartheta>\frac{\mu_2}{\mu_1+\mu_2}$.
\end{theorem}

\begin{proof}
Using the analysis process similar to that in Theorems \ref{tm1} and \ref{tm2}, it obtains from \eqref{t21} that for $t\in[0,s_0)$,
\begin{equation*}
V(t)\leq V(0)\exp\{-\mu_{1}t\}.
\end{equation*}
For $t\in[s_0,t_1)$,
\begin{equation*}
\begin{split}
V(t)&\leq V(s_0)\exp\{\mu_2(t-s_0)\}\\
&\leq V(0)\exp\{-\mu_{1}s_0+\mu_2(t-s_0)\}.
\end{split}
\end{equation*}
For $t\in[t_{1},s_{1})$,
\begin{equation*}
\begin{split}
V(t)&\leq V(t_1)\exp\{-\mu_1(t-t_1)\}\\
&\leq V(0)\exp\{-\mu_1(t-t_1+s_0)+\mu_{2}(t_1-s_0)\}.
\end{split}
\end{equation*}
For $t\in[s_{1},t_{2})$,
\begin{equation*}
\begin{split}
V(t)\leq&~V(s_1)\exp\{\mu_2(t-s_1)\}\\
\leq&~V(0)\exp\{-\mu_1(s_1-t_1+s_0)+\mu_2(t-s_1+t_1-s_0)\}.
\end{split}
\end{equation*}

By recursion, for any $t>0$, there exists a $m^*$ such that $t\in[t_{m^*},t_{m^*+1})$. For $t\in [t_{m^*},s_{m^*})$, it follows from Definition \ref{dn2} that
\begin{equation*}
\begin{split}
V(t)\leq&~V(0)\exp\{-\mu_1\mathcal{M}(t,0)+\mu_2\mathcal{M}^{\mathrm{c}}(t,0)\}.
\end{split}
\end{equation*}
where
\begin{equation*}
\begin{split}
\mathcal{M}(t,0)=&\sum_{m=0}^{m^*-1}(s_m-t_m)+t-t_{m^*},\\
\mathcal{M}^{\mathrm{c}}(t,0)=&\sum_{m=0}^{m^*-1}(t_{m+1}-s_{m}).
\end{split}
\end{equation*}
For $t\in [t_{m^*},s_{m^*})$, it also follows that
\begin{equation*}
\begin{split}
V(t)\leq&~V(0)\exp\{-\mu_1\mathcal{M}(t,0)+\mu_2\mathcal{M}^{\mathrm{c}}(t,0)\}.
\end{split}
\end{equation*}
where
\begin{equation*}
\begin{split}
\mathcal{M}(t,0)=&\sum^{m^*}_{m=0}(s_m-t_m),\\
\mathcal{M}^{\mathrm{c}}(t,0)=&\sum_{m=0}^{m^*-1}(t_{m+1}-s_m).
\end{split}
\end{equation*}
Hence, for any $t>0$, $\epsilon\in(0,\epsilon^*)$, it holds from Definition \ref{dn2} that
\begin{equation*}
\begin{split}
V(t)\leq&~V(0)\exp\{-\mu_1\mathcal{M}(t,0)+\mu_2\mathcal{M}^{\mathrm{c}}(t,0)\}\\
\leq&~V(0)\exp\{-(\mu_1\vartheta-\mu_2(1-\vartheta))t\},
\end{split}
\end{equation*}
which together with $\mu_1\vartheta-\mu_2(1-\vartheta)>0$ implies that $(\mathbf{x}^*,\mathbf{1}_N\otimes\mathbf{x}^*)$ is exponentially stable.
\end{proof}

\begin{remark}\label{rk7}
Actually, Theorem \ref{tm3} can be regarded as a discontinuous version of Theorem 1 in \cite{ye2017distributed}. The main difference is that Theorem \ref{tm3} considers the circumstance of intermittent communication, while the results in the above literature only discuss the case of continuous-time information exchanging among neighboring players. From a theoretical point of view, the proposed AIC strategy with ACR can allocate the communication time more rationally. Moreover, with regard to the advantages of using ACR rather than the quasi-periodicity assumption to deal with practical problems, refer to \cite{wang2021stabilization} and \cite{wang2021stability}.
\end{remark}

\section{Numerical experiments}\label{sec4}
In the following, we will simulate the electrical energy consumption game where the players' actions are one-dimensional, and the connectivity control game where the players' actions are two-dimensional, successively.
\begin{figure}[!htb]
\centering
\begin{tikzpicture}
[L1Node/.style={circle,draw=black!100, fill=white!10, very thick, minimum size=8mm}]

\node[L1Node] (a_1) at (2.1,3)  {1};
\node[L1Node] (a_2) at (0.5,1.8)  {2};
\node[L1Node] (a_3) at (1.15,0.2)  {3};
\node[L1Node] (a_4) at (2.95,0.2)  {4};
\node[L1Node] (a_5) at (3.7,1.8)  {5};

\draw[->] (a_1)--(a_5);
\draw[->] (a_5)--(a_4);
\draw[->] (a_4)--(a_3);
\draw[->] (a_3)--(a_2);
\draw[->] (a_2)--(a_1);				
\end{tikzpicture}
\caption{Strongly connected and directed communication graph.}\label{cg}
\end{figure}
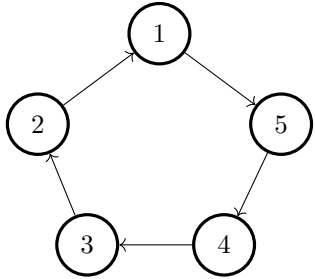

\subsection{Electrical Energy Consumption Game}
In this section, an electrical energy consumption game of five industrial players that are equipped with HVAC systems is given, whose configuration has been demonstrated in Fig. \ref{hvac}. More details of the system have been depicted fully in \cite{abrazeh2022virtual}. In the considered game, players communicate with each other via a strongly connected and directed communication graph $\mathcal{G}$ as illustrated in Fig. \ref{cg}. The payoff function of player $i$ is
\begin{equation*}
f_i(\mathbf{x})=-(x_i-x_i^q)^2-x_iW(\mathbf{x}),
\end{equation*}
where for $i\in\{1,2,3,4,5\}$, $W(\mathbf{x})=r_1\sum_{j=1}^Nx_j+r_2$, and $x_i^q$, $r_1$, $r_2$ are constants. The parameters are listed in Table \ref{ta1} \cite{ye2016game}. By direct calculation, the Nash equilibrium is calculated as $\mathbf{x}^*=(3.9379,8.6996,13.4609,18.2236,22.9854)^{\top}$, whose uniqueness can be guaranteed by Assumption \ref{ass4}. For this example, communication ratio defined in Theorem \ref{tm1} is set as $\tilde{\theta}=0.5$, quasi-periodicity assumption defined in Theorem \ref{tm2} is set as $\bar{\theta}=0.5$, $\bar{T}=10$, and ACR defined in Definition \ref{dn2} is set as $\vartheta=0.5$. Hence, the conditions in Theorems \ref{tm1}-\ref{tm3} are satisfied, respectively, by which it can be inferred that the Nash equilibrium is exponentially stable via the presented strategies for $\epsilon\in(0,\epsilon^*)$. In the numerical experiments, the initial values of all the variables in \eqref{s1} and \eqref{icl} are set as $\mathbf{x}(0)=(21,5,1,13,16)^{\top}$ and $\mathbf{y}_i(0)=(0,0,0,0,0)^{\top}$.

\begin{table}[!htb]
\caption{\\ PARAMETERS IN THE NUMERICAL EXPERIMENTS}\label{ta1}
\centering
  \renewcommand\arraystretch{1.5}
  \begin{tabular}{|c|c|c|c|c|c|}
  \hline
  ~~ & player1 & player2 & player3 & player4 & player5 \\ \hline
  $x_i^q(\mathrm{kWh})$ & 10 & 15 & 20 & 25 & 30 \\ \hline
  $r_1$ &  \multicolumn{5}{c|}{0.1} \\ \hline
  $r_2(\mathrm{MU/kWh})$ & \multicolumn{5}{c|}{5} \\ \hline
\end{tabular}
\end{table}

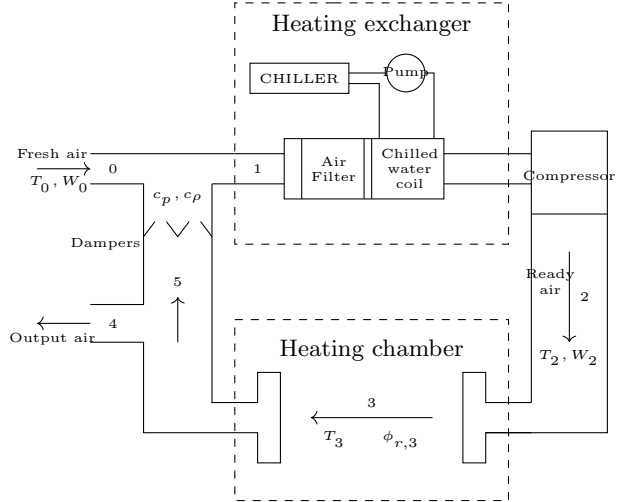
\begin{figure}[t]
\centering
\begin{tikzpicture}
\draw[dashed](0.5,0) -- (1.8,0) -- (1.8,-2.4) -- (-1.8,-2.4)--(-1.8,0)--(0.5,0);
\draw(2.1,-1.1)--(1.5,-1.1)--(1.5,-0.7)--(1.2,-0.7)--(1.2,-1.9)--(1.5,-1.9)--(1.5,-1.5)--(2.1,-1.5);
\draw(2.1,-1.1)--(2.1,2.5)--(3.1,2.5)--(3.1,-1.5)--(1.5,-1.5);
\draw(-2.1,1.8)--(-2.1,-1.1)--(-1.5,-1.1)--(-1.5,-0.7)--(-1.2,-0.7)--(-1.2,-1.9)--(-1.5,-1.9)--(-1.5,-1.5)--(-3,-1.5)--(-3,-0.3)--(-3.7,-0.3);
\draw[dashed](-0.5,4.2)--(-1.8,4.2)--(-1.8,1.0)--(1.8,1.0)--(1.8,4.2)--(-0.5,4.2);

\draw(0.95,1.6)rectangle(-1.15,2.4);
\draw(-0.92,1.6)--(-0.92,2.4) ;
\draw(-0.1,1.6)--(-0.1,2.4) ;
\draw(0.0,1.6)--(0.0,2.4) ;
\draw(0.95,1.8)--(2.1,1.8);
\draw(0.95,2.2)--(2.1,2.2);
\draw(-1.15,1.8)--(-2.1,1.8);
\draw(-1.15,2.2)--(-3.7,2.2);
\draw(-3.7,0.2)--(-3,0.2)--(-3,1.8)--(-3.7,1.8);
\draw(-1.6,3)rectangle(-0.3,3.4);
\draw(-0.3,3.13)--(0.1,3.13)--(0.1,2.4);
\draw(3.1,1.4)--(2.1,1.4);
\draw(-0.3,3.27)--(0.1,3.27)--(0.18,3.27);
\draw (0.45,3.27) circle (0.25);
\draw(0.7,3.27)--(0.82,3.27)--(0.82,2.4);
\draw [->]    (0.8,-1.3) -- (-0.8,-1.3);
\draw [->]    (-4.4,2) --(-3.7,2);

\draw [->]    (-3.7,-0.05) --(-4.4,-0.05);
\draw [->]    (-2.55,-0.3) --(-2.55,0.3);
\draw [->]    (2.6,0.9) --(2.6,-0.3);
\draw(-2.4,1.3)--(-2.55,1.1)--(-2.7,1.3);
\draw(-2.25,1.3)--(-2.1,1.1);
\draw(-2.85,1.3)--(-3,1.1);
\tiny
\node (a) at (-4.2,2.2){Fresh air};
\node (a) at (-4.2,-0.25){Output air};
\node (a) at (-3.4,-0.05){$4$};
\node (a) at (-3.4,2){$0$};
\node (a) at (-1.5,2){$1$};
\node (a) at (-4.1,1.8){$T_{_{0}},W_{_{0}}$};
\node (a) at (-0.5,2.1)  {Air};
\node (a) at (-0.5,1.9)  {Filter};
\node (a) at (0.5,2.2)  {Chilled};
\node (a) at (0.5,2)  {water};
\node (a) at (0.5,1.8)  {coil};
\node (a) at (0.45,3.27)  {Pump};
\node (a) at (-0.95,3.2)  {CHILLER};
\node (a) at (2.6,1.9){Compressor};
\node (a) at (2.35,0.6){Ready};
\node (a) at (2.35,0.4){air};
\node (a) at (2.8,0.3){$2$};
\node (a) at (-2.55,0.5){$5$};
\node (a) at (-2.55,1.6){$c_{_{p}},c_{\rho}$};
\node (a) at (0,-1.1){$3$};
\node (a) at (0,-1.6){$T_{_{3}}\qquad \phi_{_{r,3}}$};
\node (a) at (2.6,-0.5){$T_{_{2}},W_{_{2}}$};
\node (a) at (-3.5,1){Dampers};
\small
\node (a) at (0,-0.4){Heating chamber};
\node (a) at (0,3.9){Heating exchanger};
\end{tikzpicture}
\caption{Structure of a typical HVAC system \cite{abrazeh2022virtual}.}\label{hvac}
\end{figure}

Firstly, the players' actions and estimates governed by the proposed PIC strategy are plotted in Figs. \ref{picx} and \ref{picy}. Take the communication time intervals as $[0,5)\cup[10,15)\cup[20,25)\cup[30,35)\cup[40,45)\cup[50,55)\cup[60,65)\cup[70,75)\cup[80,85)\cup[90,95)$. Then, $\inf_m\{s_m-t_m\}=\bar{\theta}=0.5$, $\sup_m\{t_{m+1}-t_m\}=\bar{T}=10$. From the results, it is not difficult to find that the players¡¯ actions still converge to the Nash equilibrium at length, no matter how far the initial values are from the Nash equilibrium, which can verify Theorem \ref{tm1}.

Secondly, take the communication time intervals as $[0,6)\cup[10,14)\cup[20,24.5)\cup[29,34.5)\cup[41,46)\cup[50,55.5)\cup[60,64.5)\cup[72,76)\cup[80,85.5)\cup[90,95.5)$. Then, $\inf_m\{s_m-t_m\}=\bar{\theta}=0.5$, $\sup_m\{t_{m+1}-t_m\}=\bar{T}=10$. From Fig. \ref{aicx}, it can be seen that the initialization of the variables is the same as Figs. \ref{picx}. In the same way,
the proposed conventional AIC strategy also enables the players' actions to converge to the Nash equilibrium at length. Fig. \ref{aicy} shows that $\mathbf{y}_i$ also converge to the Nash equilibrium at length. Therefore, the conclusion in Theorem \ref{tm2} is numerically testified.

\begin{figure}[!htb]
\centering
\includegraphics[width=2.8in, height=2.1in]{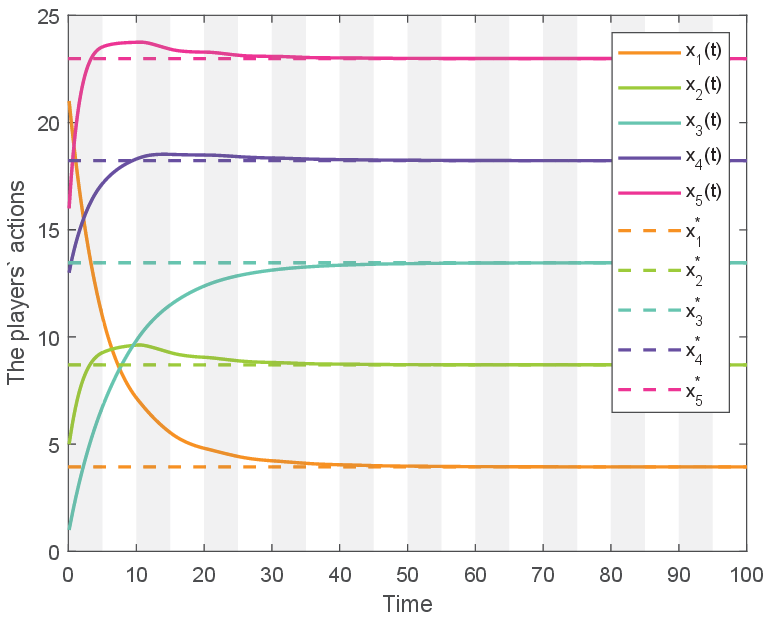}
\caption{The actions $\mathbf{x}_i$ of players generated by PIC strategy.}\label{picx}
\end{figure}

\begin{figure}[!htb]
\centering
\includegraphics[width=2.8in, height=2.1in]{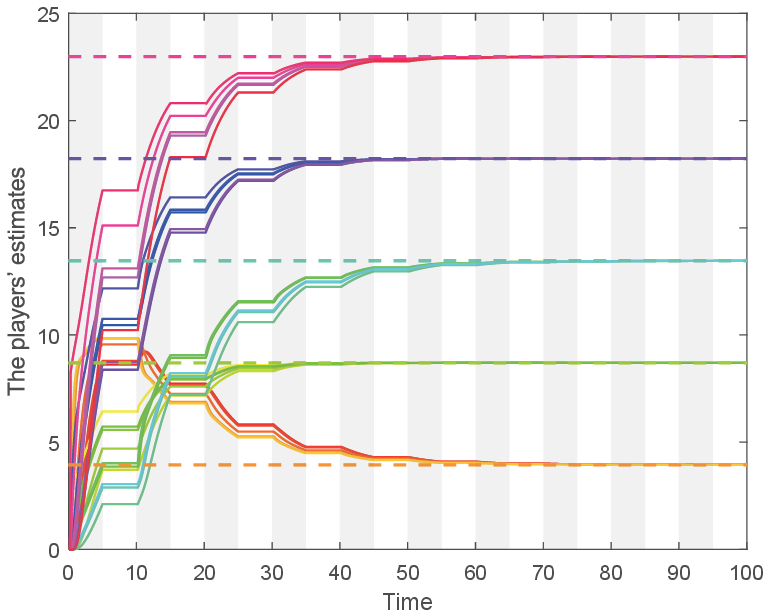}
\caption{The estimates $y_{ij}$ of all players' actions generated by PIC strategy.}\label{picy}
\end{figure}

\begin{figure}[!htb]
\centering
\includegraphics[width=2.8in, height=2.1in]{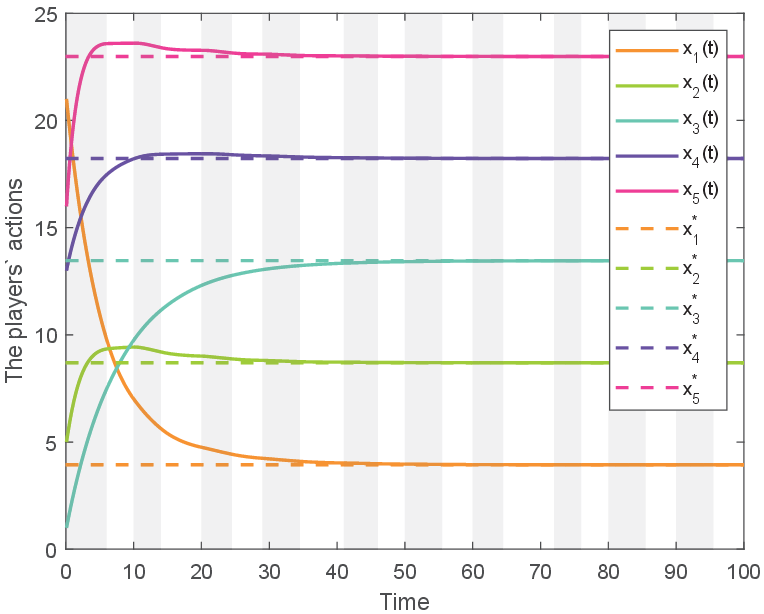}
\caption{The actions $\mathbf{x}_i$ of players generated by AIC strategy.}\label{aicx}
\end{figure}

\begin{figure}[!htb]
\centering
\includegraphics[width=2.8in, height=2.1in]{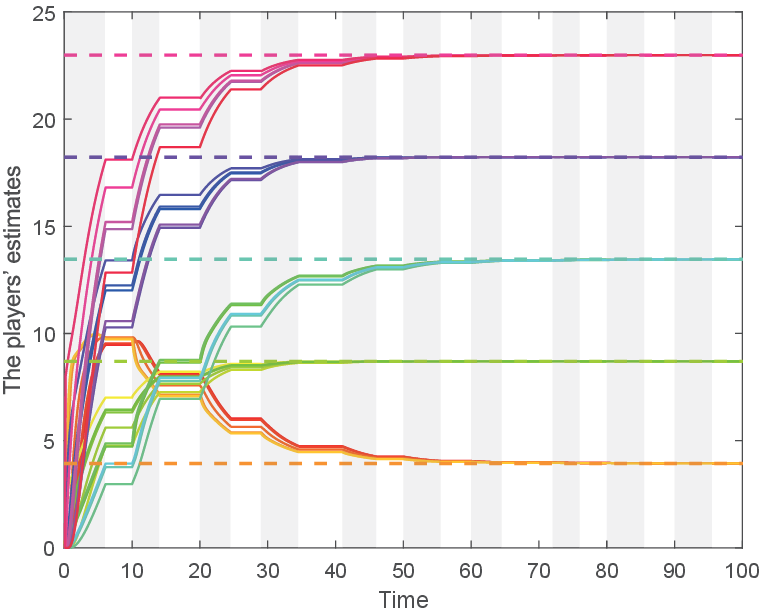}
\caption{The estimates $y_{ij}$ of all players' actions generated by AIC strategy.}\label{aicy}
\end{figure}

Thirdly, the numerical experiment results generated by the proposed AIC strategy with ACR are given in Figs. \ref{aicx1} and \ref{aicy1}. It should be pointed that communication time via conventional AIC strategy is called aperiodic, but in fact, it is quasi-periodic and need communication time widths satisfy the uniform distribution. In the following, we consider the communication time intervals as $[0,7)\cup[10,12)\cup[16,22)\cup[29,33.5)\cup[38,38.5)\cup[48,57.9)\cup[63,69)\cup[76,82)\cup[87,95)$. Obviously, the interval distribution is irregular. Then, $\inf_m\{s_m-t_m\}=\bar{\theta}=0.05$, $\sup_m\{t_{m+1}-t_m\}=\bar{T}=9.9$, $\vartheta=0.5$. The simulation results for \eqref{s1} and \eqref{icl} are shown in Fig. \ref{aicx1}, which illustrate that the players' trajectories converge to the actual Nash equilibrium at length. Moreover, the evolutions of $y_{ij}$ are shown in Fig. \ref{aicy1}. Finally, some related values for three proposed strategies are shown in Table \ref{ta2}, thus verifying the validity of the results.

\begin{figure}[!htb]
\centering
\includegraphics[width=2.8in, height=2.1in]{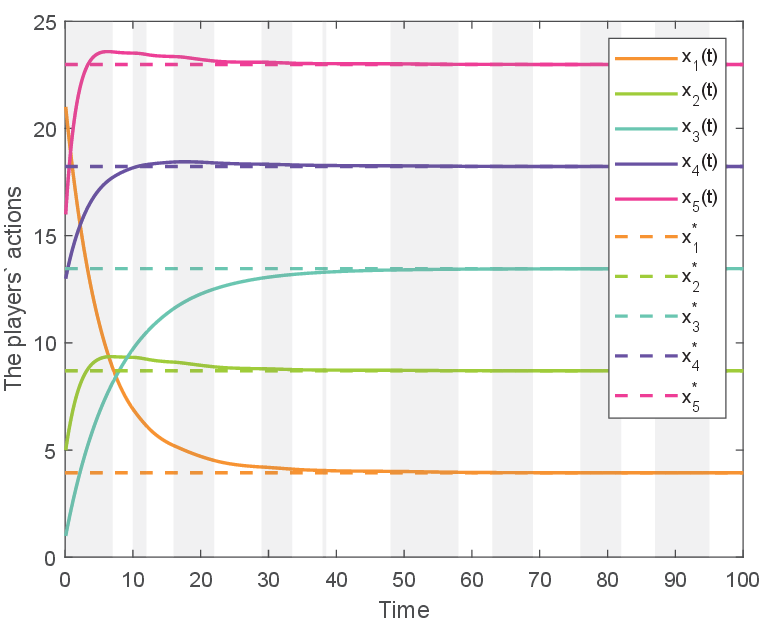}
\caption{The actions $\mathbf{x}_i$ of players generated by AIC strategy with ACR.}\label{aicx1}
\end{figure}

\begin{figure}[!htb]
\centering
\includegraphics[width=2.8in, height=2.1in]{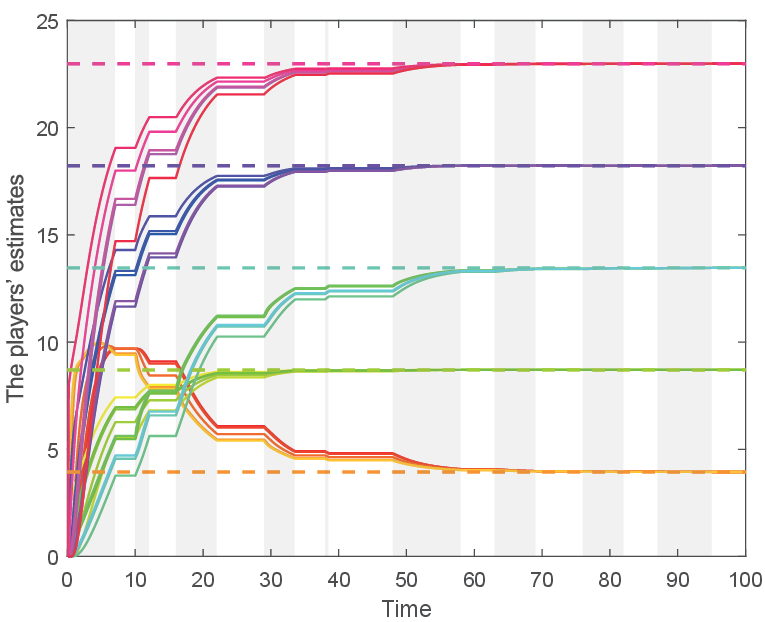}
\caption{The estimates $y_{ij}$ of all players' actions generated by AIC strategy with ACR.}\label{aicy1}
\end{figure}

\begin{table}[!htb]
\caption{\\ COMMUNICATION TIME WIDTH FOR PLAYERS IN EACH PERIOD}\label{ta2}
  \centering
  \renewcommand\arraystretch{1.5}
  \begin{tabular}{|c|c|c|c|}
  \hline
  ~~ & PIC & AIC & AIC with ACR \\ \hline
  Min interval & 5 & 4.5 & 0.5 \\ \hline
  Mean interval & 5 & 5 & 5 \\ \hline
  Max interval & 5 & 5.5 & 9.9 \\ \hline
\end{tabular}
\end{table}

\begin{remark}\label{rk9}
Recently, intermittent communication has been widely used in Renewable power prediction technology \cite{ma2021research,wu2020research,yanqi2020key}. In nature, wind and solar power generation has the characteristics of intermittency and volatility. When large-scale power generation is connected to the grid, it is easy to have an impact on the security of the grid. Therefore, as a research direction of intermittent new energy power generation and grid connection technology, new energy power prediction technology has developed rapidly. Hence, it is meaningful to consider the problem of electrical energy consumption game with intermittent communication.
\end{remark}

\subsection{Connectivity Control Game}
The generation of connectivity control game is based on the requirement of maintaining information sharing among wireless sensors in network monitoring systems. In this section, a connectivity control game \cite{ma2014distributed} is considered where the players' actions are two-dimensional. The communication graph for the players is also depicted in Fig. \ref{cg}. The players' payoff functions are
\begin{equation*}
f_i(\mathbf{x})=\hat{W}^a_i(x_i)+\hat{W}^b_i(\mathbf{x}),
\end{equation*}
where for $i\in\{1,2,3,4,5\}$, $\hat{W}^a_i(x_i)=x_i^{\top}w_{ii}x_i+x_i^{\top}w_i+i$, and $w_{ii}=\mathrm{diag}\{i,i\}$, $w_i=(i,i)^{\top}$. In particular, $k_i$ in \eqref{s1} should be less than zero due to the considered game \cite{li2022distributed}. Moreover,
\begin{equation*}
\begin{split}
&\hat{W}^b_1(\mathbf{x})=\Vert x_1-x_2\Vert^2,\\
&\hat{W}^b_2(\mathbf{x})=\Vert x_2-x_3\Vert^2,\\
&\hat{W}^b_3(\mathbf{x})=\Vert x_3-x_2\Vert^2,\\
&\hat{W}^b_4(\mathbf{x})=\Vert x_4-x_2\Vert^2+\Vert x_4-x_5\Vert^2,\\
&\hat{W}^b_5(\mathbf{x})=\Vert x_5-x_1\Vert^2.
\end{split}
\end{equation*}
For $i\in\{1,2,3,4,5\}$, $j\in\{1,2\}$, $x_{ij}^*=-0.5$ is the unique Nash equilibrium point \cite{ye2019rise}. In the following, we provide numerical results for the AIC strategy with ACR. For notational clarity, let $x_i=(x_{i1},x_{i2})^{\top}\in\mathbb{R}^{2}$, $y_{ij}=(y_{ij1},y_{ij2})^{\top}\in\mathbb{R}^{2}$. In Fig. \ref{aicx2}, $x_{i1}$ and $x_{i2}$ stand for the horizontal coordinates and vertical coordinates, respectively. In Fig. \ref{aicy2}, $y_{ij1}$ and $y_{ij2}$ stand for the horizontal coordinates and vertical coordinates, respectively. Besides, $\mathbf{x}$ and $\mathbf{y}$ can be randomly initialized from $[-15,15]$. The simulation results of the comparison produced by continuous algorithm are given in Figs. \ref{con1} and \ref{con2}.

\begin{figure}[!htb]
\centering
\includegraphics[width=3.8in, height=2.8in]{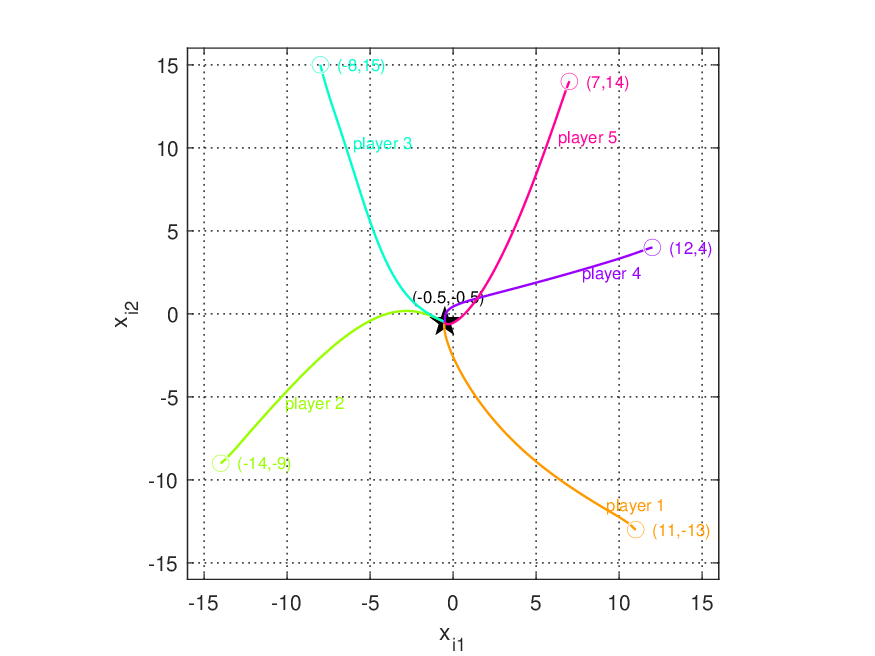}
\caption{The plots of the players' state trajectories generated by continuous algorithm.}\label{con1}
\end{figure}
\begin{figure}[!htb]
\centering
\includegraphics[width=3.5in, height=2.8in]{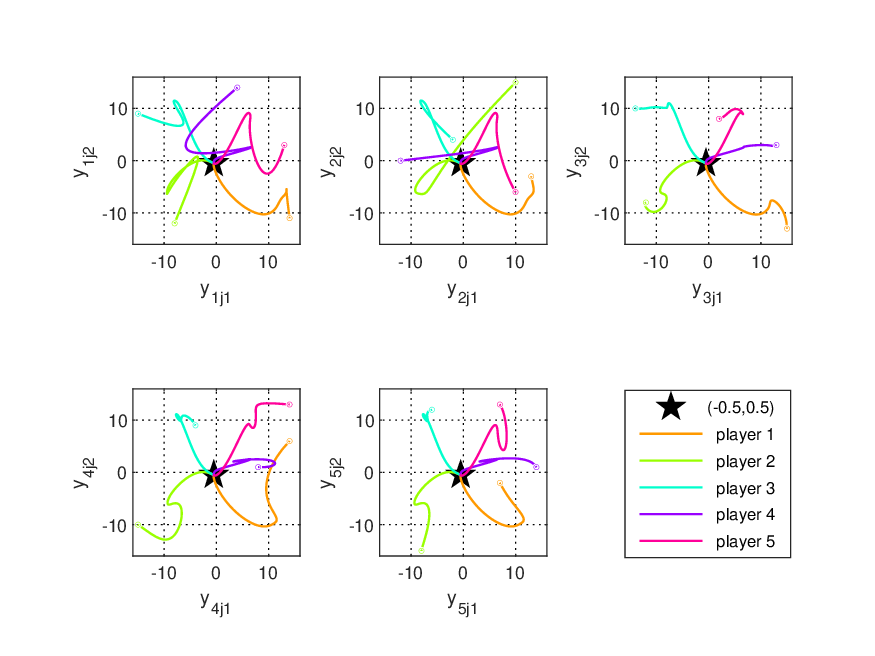}
\caption{The plots of $y_{ij1}$ versus $y_{ij2}$ generated by continuous algorithm.}\label{con2}
\end{figure}

Similarly, we consider the communication time intervals as $[0,7)\cup[10,12)\cup[16,22)\cup[29,33.5)\cup[38,38.5)\cup[48,57.9)\cup[63,69)\cup[76,82)\cup[87,95)$. In this example, we only give the verification of AIC strategy with ACR. In order to simplify the narration, the other two strategies will not be repeated. The simulation results produced by the AIC strategy with ACR $\vartheta=0.5$ in \eqref{icl} are given in Figs. \ref{aicx2} and \ref{aicy2}. From Fig. \ref{aicx2}, it is clear that the proposed AIC strategy with ACR ensures that the players' actions can converge to the Nash equilibrium at length. Fig. \ref{aicy2} illustrates that $\mathbf{y}_i$ also converge to the Nash equilibrium at length. To this end, Theorem \ref{tm3} is verified.

\begin{figure}[!htb]
\centering
\includegraphics[width=3.8in, height=2.8in]{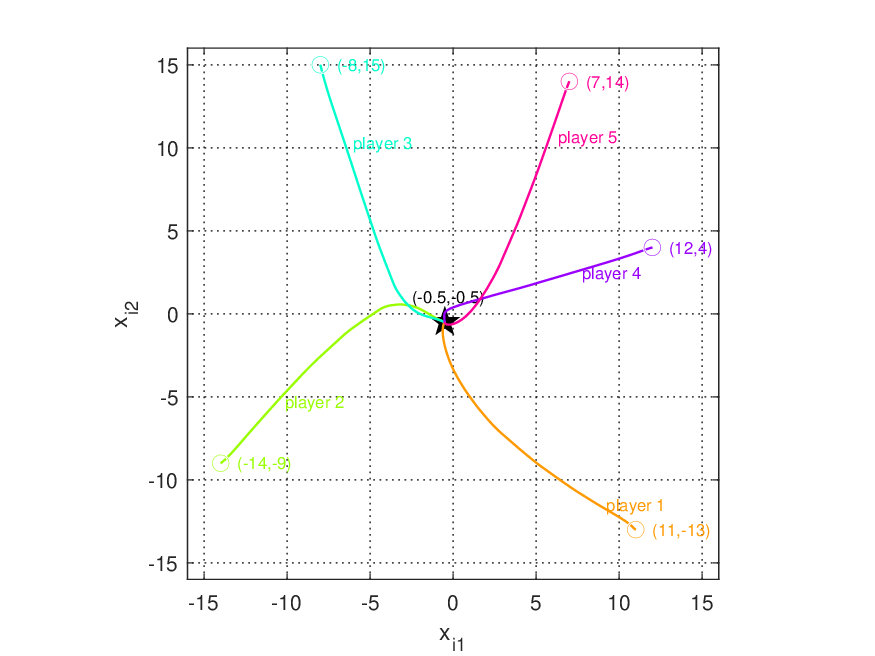}
\caption{The plots of the players' state trajectories generated by \eqref{s1}.}\label{aicx2}
\end{figure}
\begin{figure}[!htb]
\centering
\includegraphics[width=3.5in, height=2.8in]{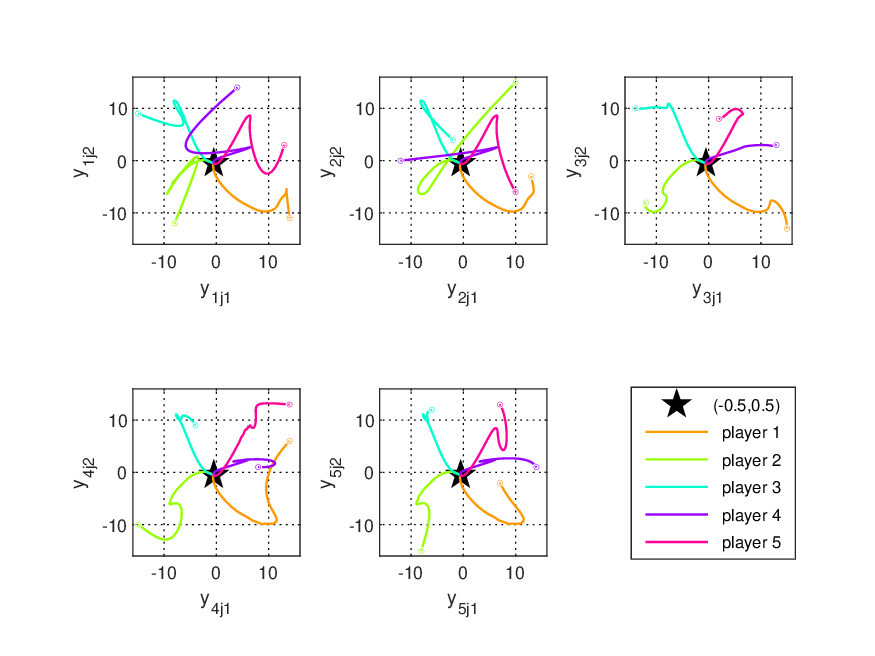}
\caption{The plots of $y_{ij1}$ versus $y_{ij2}$ generated by \eqref{icl}.}\label{aicy2}
\end{figure}

\section{Conclusions}\label{sec5}
This paper investigates Nash equilibrium seeking strategy for networked games, where the players transmit the estimate information with their neighbors intermittently. A PIC strategy, a conventional AIC strategy, and a novel AIC strategy with ACR are proposed. The PIC strategy and the conventional AIC strategy achieve exponential convergence by adjusting communication time intervals to satisfy the relevant conditions. Compared with these two strategies, the AIC strategy with ACR characterizes the distributions of communication time and silent time more reasonably. Based on the method of Lyapunov stability analysis, all strategies in this paper are theoretically validated. In addition, second-order players have drawn some attention. We remain this significant research interests as our further research fields.

\bibliography{NE}




\end{document}